\documentclass[12pt,leqno,fleqn]{amsart}  
\usepackage{amsmath,amstext,amsthm,amssymb,amsxtra}
\usepackage[top=1.5in, bottom=1.5in, left=1.25in, right=1.25in]	{geometry}
\usepackage{txfonts} 
\usepackage[T1]{fontenc}
\usepackage{lmodern}

 \usepackage{euler}   
\usepackage{pdfsync}

\usepackage{float} 

\usepackage{tikz}

\usepackage{mathtools}
\mathtoolsset{showonlyrefs,showmanualtags}

\usepackage{hyperref} 
\hypersetup{
    colorlinks=true,       
    linkcolor=blue,          
    citecolor=magenta,        
    filecolor=magenta,      
    urlcolor=cyan           
}

\usepackage[msc-links]{amsrefs}

\theoremstyle{plain} 
\newtheorem{lemma}[equation]{Lemma} 
 
\newtheorem{theorem}[equation]{Theorem} 
\newtheorem{corollary}[equation]{Corollary} 

\newtheorem{priorResults}{Theorem}

\theoremstyle{definition}

\theoremstyle{remark}

\numberwithin{equation}{section}

\title[Weak Type Estimates] {Borderline Weak Type Estimates for  Singular Integrals and Square Functions}

\author[C. Domingo-Salazar]{Carlos Domingo-Salazar}
\address{Departament de Matem\`atica Aplicada i~An\`alisi, Universitat de Barcelona, 08071 Barcelona, Spain.} \email{domingo@ub.edu}
\thanks{The first author was supported by grants MTM2013-40985-P and 2014SGR289}

\author[M. Lacey]{Michael Lacey}   

\address{ School of Mathematics, Georgia Institute of Technology, Atlanta GA 30332, USA}
\email {lacey@math.gatech.edu}
\thanks{The second author was supported in part by grant NSF-DMS 1265570. }

\author[M. Lacey]{Guillermo Rey}   

\address{Department of Mathematics, Michigan State University, East Lansing USA}
\email{reyguill@math.msu.edu}

\begin{document}

\begin{abstract}
For any Calder\'on-Zygmund operator $ T$, any weight $ w$, and $ \alpha >1$, the operator $ T$ is bounded 
as a map from $ L ^{1} (M _{ L \log\log L  (\log\log\log L) ^{\alpha } } w )$ into weak-$L^1(w)$.  
The interest in questions of this type goes back to the beginnings of the weighted theory, 
with prior results, due to  Coifman-Fefferman, P\'erez,  and Hyt\"onen-P\'erez, on the $ L (\log L) ^{\epsilon }$ scale. 
Also, for square functions $ S f$, and weights $ w \in A_p$, the norm of $ S$ from $ L ^p (w)$ to weak-$L^p (w)$, $ 2\leq p < \infty $,   is 
bounded by $ [w] _{A_p}^{1/2} (1+\log [w] _{A_ \infty }) ^{1/2}  $, which is a sharp estimate.  
\end{abstract}

	\maketitle  
	
\section{Introduction and main results} 

We are interested in two different weak-type estimates for singular integrals and square functions, 
at critical endpoint cases.  The first of these is the weak $ L ^{1}$-endpoint for singular integrals. 
In \cite{MR0284802}, C. Fefferman and E. M. Stein used  the estimate  
\begin{equation}\label{e:FS}
\sup_{\lambda>0} \lambda w(Mf>\lambda)\leq C\int_{\mathbb R^d} |f(x)|\,Mw(x)dx,
\end{equation}
where $w$ is a general weight on $\mathbb R^d$ and $ M$ is the  Hardy-Littlewood maximal operator, 
to study  vector-valued inequalities for $ M$.  Indeed, the connection between weighted inequalities 
and their vector valued extensions was intensively studied in the next decades (see 
for instance \cites{MR736248,MR0133653}).  
Rubio de Francia's extrapolation \cites{MR745140,MR663793} shows that from the inequality above, one can 
obtain all the $ A_p$ inequalities for the maximal function. See \cite{MR2797562} for a recent accounting of that 
theory, especially \cite{MR2797562}*{Cor 3.10} being relevant to the extrapolation results that follow from our main theorem. 
A  conjecture of B.~Muckenhoupt and R.~Wheeden \cite{problems}*{p. 134}  claimed that the inequality above should be true with the maximal function on the left-hand side replaced by a  singular integral operator. This conjecture was disproved by Reguera \cite{MR2799801} and Reguera-Thiele \cite{MR2923171}. Also see \cite{13125255}.  

The focus here is on the positive direction.  Inequality \eqref{e:FS} for singular integral operators is true if $ Mw$ is replaced by a larger maximal function.   
For instance, one can replace $ M w $ by $(M w ^{r}) ^{1/r}$, for exponents $ 1< r < \infty $, first proved by 
Coifman-Fefferman \cite{MR0358205}, from which you can derive the result of Benedeck-Calder\'on-Panzone \cite{MR0133653}.  
Finer variants can be had by considering Orlicz versions of the maximal function.  
In 1994, 
C. P\'erez \cite{MR1260114} already established a version of the inequality with 
$ (M w ^{r}) ^{1/r}$ replaced by $ M _{L (\log L) ^{\epsilon }} w$, with $0<\epsilon<1$.  Closely related to this is the 
sharp control of a Calder\'on-Zygmund operator on $ L ^{p} (w)$, relative to the $ A_1$ constant of the 
weight, a theme of Fefferman-Pipher \cite{MR1439553}. 
This matter was also pursued in,
for instance, \cites {MR2427454,MR2480568}.  Recently, T. Hyt\"onen and C. P\'erez  quantified the estimates  of Coifman-Fefferman  \cite{MR0358205}  and P\'erez \cite{MR1260114}.  

\begin{priorResults}[Hyt\"onen-P\'erez, \cite{MR3327006}]\label{t:HP} For all $ 0 < \epsilon < 1$ and Calder\'on-Zygmund operator $ T$, 
\begin{equation}  \label{e:eps}
\sup _{\lambda>0 } \lambda w ( T^\ast  f > \lambda ) \lesssim \frac 1 \epsilon \int_{\mathbb R^d} \lvert  f(x)\rvert \, M _{L (\log L) ^{\epsilon }} w(x) dx.  
\end{equation} 
In particular, for every $1<r<\infty$, 
\begin{equation}  \label{e:Mr}
\sup _{\lambda>0 } \lambda w ( T^\ast  f > \lambda ) \lesssim (1+\log r') \int_{\mathbb R^d} \lvert  f(x)\rvert \, M _r w(x) dx,  
\end{equation} 
where $M_r w:= (Mw^r)^{1/r}$, and if $ w $ is an $ A_1$ weight, there holds 
\begin{equation}\label{e:logA}
\sup _{\lambda>0 } \lambda w ( T^\ast  f > \lambda ) \leq  C[w]_{A_1}\log  (1 + [w] _{A_ \infty }) \int_{\mathbb R^d} \lvert  f(x)\rvert \,w(x) \,dx.  
\end{equation}
\end{priorResults}

Recall that $A_p$ weights, $ 1 < p < \infty $ are those non-negative weights $ w$ satisfying 
\begin{equation*}
[ w ] _{A_p}  = \sup _{Q} \frac 1{ \lvert  Q \rvert}  \int _{Q} w \; dx \cdot  
\Bigl( \frac 1{ \lvert  Q \rvert}  \int _{Q} w ^{-\frac 1 {p-1}}\; dx \Bigr)^{p-1}< \infty , 
\end{equation*}
where the supremum is over cubes in $ \mathbb R ^{d}$.  
As $ p$ decreases to $ 1$, the condition above strengthens to 
\begin{equation*}
[ w ] _{A_1} = \sup _{x\in\mathbb R^d}  \frac {M w (x)} {w (x)} < \infty, 
\end{equation*}
and as $ p$ tends to $ \infty $, the condition weakens to 
\begin{equation*}
[w ]_{A _{\infty }} = \sup _{Q} \frac { \int _{Q} M (w \mathbf 1_{Q})} {w (Q)}.  
\end{equation*}
The estimate for $M_r$, \eqref{e:Mr},  is not explicitly mentioned in \cite{MR3327006} but it can be derived from \eqref{e:eps} using the optimization argument in \cite{MR3327006}*{Cor. 1.4}. The result for $A_1$ weights then follows by an appropriate choice of $r>1$ based on the sharp reverse H\"older's inequality for the $A_\infty$ constant \cite{MR3092729}*{Thm. 2.3}, recalled in \eqref{e:srh} below.  
\smallskip

Here $T ^{\ast }$ is the maximal truncation of $T$. Precise definitions of some standard objects, like 
Calder\'on-Zygmund operators, are given in section 2. 
The operator 
$M_{L (\log L) ^{\epsilon }}$ is the modified Hardy-Littlewood maximal operator with respect to the function $\varphi(t)=t(1+\log_+t)^\epsilon$. More precisely,
\begin{equation*}
M_{\varphi (L)}w(x):= \sup_{x\in Q} \|w\|_{\varphi (L),Q},
\end{equation*}
and
\begin{equation*}
\|w\|_{\varphi (L),Q}= \inf\left\{\lambda>0:\frac{1}{|Q|}\int_Q \varphi\left(\frac{w(x)}{\lambda}\right)dx\leq 1\right\}.
\end{equation*}
Notice that if $ r\geq 1$ and $\varphi(t)=t ^{r}$, then $M_{\varphi (L)} w = ( M w ^{r} ) ^{1/r}= M_r w$. 
Finally, let us recall some properties of Young functions and Orlicz spaces (see \cite{orlicz} for more details).  Throughout the paper, $\varphi:[0,\infty)\rightarrow [0,\infty)$ will be a \emph{Young function}, that is, a convex, increasing function such that $\varphi(0)=0$ and $\lim_{t\rightarrow\infty}\varphi(t)=\infty$. From these properties, on can deduce that its inverse $\varphi^{-1}$ exists on $(0,\infty)$. Moreover, given a Young function, we can define its complementary function $\psi$ by
\begin{equation*}
\psi(s)=\sup_{t>0}\,\{st-\varphi(t)\}.
\end{equation*}
We will assume that $\lim_{t\rightarrow \infty} \varphi(t)/t=\infty$ to ensure that $\psi$ is finite valued. Under these conditions, $\psi$ is also a Young function and it is associated with the dual space of $\varphi(L)$. More precisely, one has the following generalized H\"older's inequality: 
\begin{equation}\label{e:holder}
\frac{1}{|Q|}\int_Q |f(x)g(x)|dx\lesssim \|f\|_{\varphi(L),Q}\|g\|_{\psi(L),Q}.
\end{equation}

The main result pushes the prior result to log-log scale as reflected in the tower of height 2 in \eqref{e:suff}.

\begin{theorem}\label{t:borderline}  
Suppose the Young function $ \varphi $ satisfies 
\begin{equation}\label{e:suff}
c_\varphi = \sum_{k=1} ^{\infty }   \frac {1 } { \psi^{-1} (2 ^{2 ^{k}})}    < \infty. 
\end{equation}
Then,   for all Calder\'on-Zygmund operators $ T$, and any weight $ w$ on $ \mathbb R ^{d}$, it holds that
\begin{equation}\label{e:weak}
\sup_{\lambda >0} \lambda w\{  T^\ast f > \lambda \} \lesssim c_\varphi   \int_{\mathbb R^d} \lvert  f(x)\rvert\, M _{\varphi(L)} w(x) \; dx.
\end{equation}
\end{theorem}

We will see that this result contains Theorem~\ref{t:HP}.  It also contains new estimates such as those presented in Corollary \ref{c:borderline} below. 
Notice that in \eqref{e:iii}, we very nearly have a double log in the maximal function, a possibility that 
was alluded to by Hyt\"onen-P\'erez \cite{MR3327006}. For simplicity, from now on we will adopt the following notation:
\begin{equation*}
\log_1(x):=1+\log_+(x) \quad \text{and}\quad \log_k(x):=\log_1\log_{k-1}(x), \,\, \text{for } k>1.
\end{equation*}

\begin{corollary}\label{c:borderline}
Under the assumptions of Theorem \ref{t:borderline}, for $1<\alpha<2$ it holds that,


\begin{equation} \label{e:ii}
\sup _{\lambda >0} \lambda w \{  T^\ast f > \lambda \} \lesssim \frac 1 {\alpha -1}\int_{\mathbb R^d} \lvert  f(x)\rvert\, M _{ L (\log_2 L) ^{\alpha } } w(x) \; dx,
\end{equation}
and in fact, one can also reach
\begin{equation} \label{e:iii}
\sup _{\lambda >0} \lambda w \{  T^\ast f > \lambda \} \lesssim \frac 1 {\alpha -1}\int_{\mathbb R^d} \lvert  f(x)\rvert \,M _{ L \log_2 L  (\log_3 L)^{\alpha} } w(x) \; dx  . 
\end{equation}

\end{corollary}
\medskip

Most of the prior arguments \cites{MR3092729, MR3327006, MR1260114, MR2427454} depend upon extrapolation 
type arguments, namely establishing a range of $ L ^{p}$ inequalities, and then making an appropriate choice of $ p \approx 1$ to conclude the argument at $ L ^{1}$. 
We address these two  points in the (short) proof of Theorem~\ref{t:borderline}.    
\begin{itemize}
\item One should work directly with the weak-type norm, avoiding a Calder\'on-Zygmund decomposition approach involving $L^p$ estimates for some $p>1$.   
This is addressed by our decomposition of the sparse operator based upon the function $ f$. 

\item The quasi-norm nature of the weak-$ L ^{1}$ norm is  accounted for by using a restricted weak-type approach, and the specific structure of the operators in question. 

\end{itemize}

Concerning sharpness, it seems very likely that  Theorem~\ref{t:borderline} does not have any essential strengthening. 
But, the main counterexamples have at their heart a counterexample to an $ L ^2 $ inequality 
for martingale transforms,  as in Reguera's first paper on the subject \cite{MR2799801}.   Perhaps one could rethink the counterexamples  for sparse operators.  They must exist, but they seem somewhat involved to construct directly.  

\bigskip 

We turn to our second result, which concerns square functions.  Define the 
intrinsic square function of M.~Wilson  \cite{wilson1} $ G _{\alpha }$, for $ 0< \alpha < 1$ by 
\begin{align*}
G_{\alpha}f(x) ^2 =  \int_{\Gamma(x)} A_{\alpha}f(y,t)^2   \frac {dy dt} {t ^{n+1}} 
\end{align*}
where  $\Gamma(x):= \{ (y,t) \in \mathbb R ^{n+1}_+ \;:\; \lvert  y\rvert < t \}$ is the cone of aperture one in the upper-half plane, 
and 
\begin{align*}
A_{\alpha}f(x,t) &= \sup_{\gamma \in C_{\alpha}} |f \ast \gamma_t(x)|
\end{align*}
where $\gamma_t(x) = t^{-n} \gamma(x t^{-n})$ and  $C_{\alpha}$ is the collection of functions $\gamma$ supported in the unit ball with mean zero and such that $| \gamma(x) - \gamma(y) | \le |x - y|^{\alpha}$.  
The sharp weighted strong type norms for the square function have a critical case at $ p=3$, see Lerner \cite{MR2770437}.

We are concerned with the weak-type bounds, which have a critical case of $ p=2$, at which a $ (\log_1 [w] _{A_ \infty }) ^{1/2} $ 
appears in the sharp estimate. 

\begin{theorem}\label{t:sq} 
For $ 1\leq p < \infty $, and any weight $ w\in A_p$,  there holds 
\begin{gather*}
\lVert G _{\alpha } f\rVert _{L ^{p, \infty } (w)} \leq C_p (w) \lVert f\rVert _{L ^{p} (w)} 
\\
\textup{where} \qquad 
C_p (w) = 
\begin{cases}
[w] _{A_p}^{1/p}  & 1\leq p < 2 
\\
[w] _{A_p} ^{1/2} ( \log_1 [w] _{A_ \infty }) ^{1/2}  & 2\leq p < \infty 
\end{cases}
\end{gather*}
\end{theorem}

In the next section, we recall some basic facts, including the reduction to sparse operators. 
Theorem~\ref{t:borderline} is then proved, followed by the Theorem~\ref{t:sq}.

\section{Background} 

To say that $ T$ is a  Calder\'on-Zygmund operator is to say that for $ f, g$ Schwartz functions on $ \mathbb R ^{d}$, 
with a positive distance between their supports, there holds 
\begin{equation*}
\langle T f, g \rangle = \int_{\mathbb R^d} \int_{\mathbb R^d} K (x,y) f (y) g (x) \;dx\,dy,
\end{equation*}
where $ K \;:\; \mathbb R ^{d} \times \mathbb R ^{d} \mapsto \mathbb R $ satisfies the size and smoothness conditions 
\begin{gather*}
\lvert  K (x,y)\rvert \leq \frac 1 {\lvert  x-y\rvert^d }, \qquad x\neq y,
\\
\lvert  K (x,y) - K (x',y)\rvert \leq \frac {\omega \bigl(\frac {\lvert  x-x'\rvert } {\lvert  x-y\rvert } \bigr)} {\lvert  x-y\rvert ^{d} }, 
\qquad 2 \lvert  x-x'\rvert \leq \lvert  x-y\rvert  , 
\end{gather*}
and the same inequality with the roles of the variables reversed also holds.  Here the function      $ \omega \;:\; [0,1]\mapsto [0,1] $ 
is a \emph{Dini modulus of continuity}, that is a decreasing function such that 
\begin{equation*}
\int_0^1\omega(t)\frac{dt}{t}<\infty.
\end{equation*}
We furthermore assume that $ T$ extends to a bounded operator on $ L ^2 (\mathbb R ^{n})$, with norm at most one. 
As usual, we will define the maximal truncation $T^*$ by
\begin{equation*}
T^*f(x):=\sup_{\delta>0}\left|\int_{|x-y|>\delta}K(x,y)f(y)\,dy\right|.
\end{equation*}
In this setting, a result in \cite[Theorem 5.2]{150105818} provides the pointwise domination of both $Tf$ and $T^*f$ by a sum of, at most, $3^d$ sparse operators, adapted to (shifted) dyadic grids.

A \emph{sparse operator} is of the form 
\begin{equation*}
Tf = \sum_{Q\in \mathcal S} \langle f \rangle_Q \mathbf 1_{Q}, 
\end{equation*} 
where the collection $ \mathcal S $ consists of dyadic cubes which are sparse, in the sense that 
for all $ Q\in \mathcal S$, there holds 
\begin{equation} \label{e:sparse} 
\Bigl\lvert  \bigcup _{ Q'\in \mathcal S \;:\; Q'\subsetneq Q } Q' \Bigr\rvert\leq 8^{-1} \lvert  Q\rvert.   
\end{equation}
With this domination result at hand, we can restrict our attention to sparse operators in proving Theorem~\ref{t:borderline}. 

Concerning square functions, a variant of the argument in \cite{150105818} shows that the intrinsic 
square function is dominated by a sum of at most $ 3 ^{d}$ sparse square functions defined by 
\begin{equation} \label{e:Sdef}
Sf ^2  = \sum_{Q\in \mathcal S} \langle f \rangle_Q  ^2 \mathbf 1_{Q}.  
\end{equation}
And so, to prove Theorem~\ref{t:sq}, it suffices to prove the same estimate for the square functions $S f $.  
\medskip

This elementary lemma on the $\varphi (L)$-norm will be needed.  

\begin{lemma}\label{l:orlicz} Given a cube $Q\subset \mathbb R^d$, suppose that $ w \;:\; E \mapsto [0, \infty )$ 
 where $ E\subset Q$.  Then, 
\begin{equation}\label{e:norm>}
\langle w\rangle_Q=\frac{1}{|Q|}\int_Qw(x)\;dx\lesssim \frac{\lVert w \rVert _{{\varphi (L)},Q }} { \psi^{-1} (|Q|/|E|)}.
\end{equation}

\end{lemma}

\begin{proof}
This is just H\"older's inequality \eqref{e:holder} applied to $w=w\mathbf 1_{E}$ and the computation of $\|\mathbf 1_E\|_{\psi(L),Q}$. 
\end{proof}

A final remark is that we will repeatedly use the (very easy)  sharp weak-type $ A_p$ estimate for the maximal function 
\begin{equation*}
\lambda ^{p}w \bigl( M f > \lambda  \bigr) \lesssim  [w] _{A_p} \lVert f\rVert _{L ^{p} (w)}^p. 
\end{equation*}

Our analysis is entirely elementary, except for an appeal to the sharp reverse-H\"older estimate of Hyt\"onen-P{\'e}rez \cite{MR3092729}*{Thm. 2.3}. 

\begin{theorem}\label{t:srh}  There is a dimensional constant $ c>0$ so that for $ w \in A _ \infty $, 
and $ r (w) = 1+ c [w] _{A_ \infty }$,  there holds 
\begin{equation}\label{e:srh}
\langle  w ^{r (w)} \rangle_Q ^{1/ r (w)} \leq 2 \langle w \rangle_Q , \qquad \textup{$ Q$ a cube.} 
\end{equation}
\end{theorem}

\section{Proof of  Theorem \ref{t:borderline} and Corollary \ref{c:borderline}} 

Let $f$ be a non-negative function. Due to linearity of the weak-type estimate in $ \lambda $, it suffices to show that 
\begin{equation}
w ( 4  <  T f \leq 8 ) \lesssim \int_{\mathbb R^d} f(x) \,M _{\varphi (L)} w(x) \; dx . 
\end{equation}
Let $ \mathcal E = \{ 4  < T f \leq 8 \} \setminus \{ M f > 2 ^{-2}\}$. In view of the classical inequality \eqref{e:FS} of Fefferman-Stein, it is enough to check that 
\begin{equation}
\label{e:4}
w (\mathcal E) \leq \tfrac 14 \int _{\mathcal E} T f(x) \,w(x) \;dx  \lesssim \int_{\mathbb R^d} f(x)\, M _{\varphi (L)} w(x) \; dx . 
\end{equation}

By getting rid of the set $\{ Mf>2^{-2}\}$, we can eliminate from $ \mathcal S$ all those cubes $ Q$ such that $ \langle f \rangle_Q > 2 ^{-2}$.  
For $ k\ge2$, define $ \mathcal S _{k}$ to be those $ Q\in \mathcal S$ for which $ 4^{-k-1}<\langle f \rangle_Q \leq 4 ^{-k}$, and set
\begin{equation*}
T_kf=\sum_{Q\in \mathcal S_k} \langle f \rangle_Q \mathbf 1_Q.
\end{equation*}  
The key lemma is the following: 

\begin{lemma}\label{l:basic}  For each integer $ k\geq 1$, there is an absolute constant $ C$ such that, 
\begin{equation*}
\int _{\mathcal E} T _{k} f(x)\, w(x) \;dx \leq 
2 ^{-k}  w (\mathcal E) + 
\frac C {\psi^{-1} ( 2 ^{2 ^{k}})}  \int_{\mathbb R^d} f(x)\, M _{\varphi (L)} w(x) \; dx . 
\end{equation*}

\end{lemma}

It is clear that this lemma completes the proof of our main theorem. We just write $Tf=\sum_{k=1}^\infty T_kf$, and from \eqref{e:4}: 
\begin{align*}
w (\mathcal E) &\leq \tfrac 14\sum _{k=1} ^{ \infty  } \left( 2 ^{-k}  w (\mathcal E) +\frac C {\psi^{-1} ( 2 ^{2 ^{k}}))} \int_{\mathbb R^d} f(x)\, M _{\varphi (L)} w(x) \; dx \right)
\\
& \leq \tfrac 12 w (\mathcal E) + C \cdot 
c_\varphi\int_{\mathbb R^d} f(x)\, M _{\varphi (L)} w(x) \; dx,
\end{align*}
which yields Theorem \ref{t:borderline}. Let us prove the lemma:

\begin{proof}
Write $ \mathcal S _k$ as the union of  $ \mathcal S _{k,v}$, for $ v=0,1 ,\dotsc,$, where $ \mathcal S _{k,0}$ 
are the maximal elements of $ \mathcal S_k$, and $ \mathcal S _{k, v+1} $ are the maximal elements of 
$ \mathcal S _{k} \setminus \bigcup _{\ell=0} ^{v}\mathcal S _{k,\ell}$.  We are free to assume that $ \mathcal S _{k,v} = \emptyset $ 
if $ v > 4 ^{k+1}$, since for a cube $Q$ in such a family, one would have that $\mathcal E\cap Q=\emptyset$.   Hence, we need to estimate 
\begin{equation} \label{e:Tw<}
\sum_{v=0} ^{4^{k+1}}  \sum_{Q\in \mathcal S _{k,v}}\langle f \rangle_Q w (\mathcal E \cap Q).  
\end{equation}

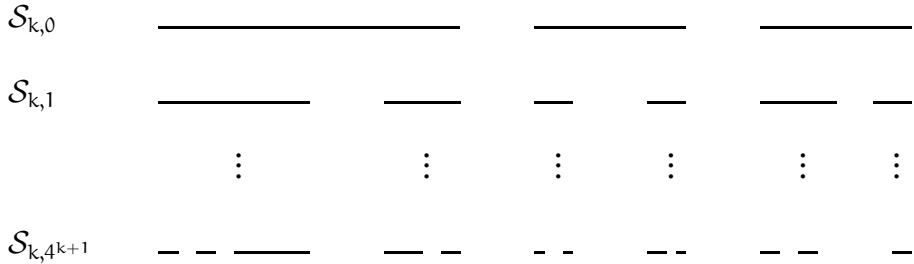
\begin{figure}[H]
\begin{center}
\setlength{\unitlength}{1cm}
\begin{picture}(10,5)
\thicklines

\put(-2,4){$\mathcal S_{k,0}$} 
\put(0,4){\line(1,0){4}}
\put(5,4){\line(1,0){2}}
\put(8,4){\line(1,0){2}}

\put(-2,3){$\mathcal S_{k,1}$} 
\put(0,3){\line(1,0){2}}
\put(3,3){\line(1,0){1}}
\put(5,3){\line(1,0){.5}}
\put(6.5,3){\line(1,0){.5}}
\put(8,3){\line(1,0){1}}
\put(9.5,3){\line(1,0){.5}}

\put(1,2){$\vdots$} 
\put(3.5,2){$\vdots$}
\put(5.25,2){$\vdots$}
\put(6.75,2){$\vdots$}
\put(8.5,2){$\vdots$}
\put(9.75,2){$\vdots$}

\put(-2,1){$\mathcal S_{k,4^{k+1}}$}  
\put(0,1){\line(1,0){.25}}
\put(.5,1){\line(1,0){.25}}
\put(1,1){\line(1,0){1}}
\put(3,1){\line(1,0){.5}}
\put(3.75,1){\line(1,0){.25}}
\put(5,1){\line(1,0){.125}}
\put(5.375,1){\line(1,0){.125}}
\put(6.5,1){\line(1,0){.25}}
\put(6.875,1){\line(1,0){.125}}
\put(8,1){\line(1,0){.25}}
\put(8.5,1){\line(1,0){.25}}
\put(9.75,1){\line(1,0){.25}}
\end{picture}
\end{center}

\vspace{-0.5cm}
\caption{Layer decomposition of $\mathcal S_k$.}
\end{figure}

Define
\begin{equation*}
 E_Q = Q \setminus \bigcup _{Q'\in \mathcal S _{k,v+1}}Q',  \quad\text{  for }  Q\in \mathcal S _{k,v}.
 \end{equation*} 
 These sets are disjoint in $ Q \in \mathcal S_k$ and using that $4^{-k-1}\leq \langle f\rangle_Q\leq 4^{-k}$ together with the sparsity condition \eqref{e:sparse}, one can check that
 \begin{align}\label{e:EQ}
\int_{Q}f\leq \frac{8}{3}\int_{E_Q}f.
 \end{align}
Set $ u=  2 ^{k} $.  
It follows again from sparsity that for each $ v \ge 0$, and $ Q\in \mathcal S _{k,v}$, 
\begin{equation}\label{e:spare-k} 
\lvert Q _{u} \rvert \leq 8 ^{-u} \lvert  Q\rvert, \qquad 
Q_u := \bigcup _{\substack{Q'\in \mathcal S _{k, v + u}\\ Q'\subset Q }} Q' .  
\end{equation}  
In \eqref{e:Tw<}, for each $ Q\in \mathcal S _{k,v}$ we decompose the set $ \mathcal E \cap Q$  into
\begin{equation*}
\mathcal E \cap Q =\mathcal  E \cap \Bigl(Q_u \cup \bigcup _{\ell =0} ^{u-1} 
\bigcup _{\substack{Q'\in \mathcal S _{k,v + \ell }\\ Q'\subset Q }} E_ {Q'} \Bigr).
\end{equation*}
Using estimates  \eqref{e:EQ}, \eqref{e:spare-k} and Lemma \ref{l:orlicz}, 
for $ Q\in \mathcal S _{k,v}$, 
\begin{align*}
\langle f \rangle_Q w (\mathcal E\cap Q_u)  & \lesssim  \int _{E_Q} f(x) \; dx  \,\langle w\mathbf 1_{\mathcal E\cap Q_u}\rangle_Q
\\
& \lesssim  \frac 1 {\psi^{-1} ( 2 ^{2 ^{k}})} \int_{E_Q} f(x) \,M _{\varphi (L)} w(x) \; dx. 
\end{align*}
This is the way in which $ M _{\varphi (L)} w$ is introduced. 
We sum this over $ v$ and $ Q\in \mathcal S _{k,v}$ to the main estimate, using the disjointness of the sets $ E_Q$:  
\begin{equation}\label{e:mainTerm}
\sum_{v=0} ^{4 ^{k+1}} \sum_{Q\in \mathcal S _{k,v}}  \langle f \rangle_Q w (\mathcal E \cap Q_u)   
\lesssim 
 \frac 1 {\psi^{-1} ( 2 ^{2 ^{k}})} \int_{\mathbb R ^{d}} f(x) \,M _{\varphi (L)} w(x) \; dx. 
\end{equation}
Now, the remaining sum is 
\begin{equation*}
\sum_{v=0} ^{4 ^{k+1}} \sum_{Q\in \mathcal S _{k,v}} 
\sum_{\ell =0} ^{ u-1}  \sum_{ \substack{Q'\in \mathcal S _{k,v+ \ell }\\ Q' \subset Q }}
\langle f \rangle_Q  w (\mathcal E  \cap E _{Q'})    
\leq 4^{-k} u \sum_{v=0}^{4^{k+1}}\sum_{Q\in \mathcal S_{k,v}}w(\mathcal E\cap E_{Q})
 \leq 2 ^{-k} w (\mathcal E).  
 \end{equation*}
This only depends on the fact that $ \langle f \rangle_Q\leq 4^{-k}$ and 
the disjointness of the sets $ E _{Q'}$. The proof of the lemma is now complete.

\end{proof}

Finally, let us check that from here we can deduce both Theorem \ref{t:HP} and Corollary \ref{c:borderline}. We have three functions $\varphi$ that we want to study:
\begin{description}
\item[1]$\varphi (t) = t(\log_1 t) ^{\epsilon }$, with $0<\epsilon<1$,
\item[2]  $\varphi (t) = t(\log_2 t) ^{\alpha }$, with $1<\alpha<2$,
\item[3] $\varphi (t) = t\log_2  t(\log_3 t) ^{\alpha }$, with $ 1<\alpha<2$.
\end{description} 
In these cases, $\varphi(t)=tL(t)$, where $L$ is the corresponding logarithmic part, and it holds that $L(t)\lesssim \psi^{-1}(t)$. Since $\psi^{-1}$ is increasing, it suffices to check that
$$
\psi(L(t))=\sup_{0<s<t}\left\{s(L(t)-L(s))\right\}\lesssim t.
$$
This follows by the mean-value theorem and the fact that $rL'(r)\leq C$, with $C>1$ being a universal constant. 
Therefore, we can work with $c_\varphi$ with $\psi^{-1}$ replaced by $L$.

\medskip 

\smallskip 
\textbf{Case 1.}  This corresponds to the inequality \eqref{e:eps}. 
From \eqref{e:suff} and taking $L(t)=(\log_1 t)^\epsilon$ instead of $\psi^{-1}$,
\begin{align*}
c _{\varphi } \simeq \sum_{k=1 }^\infty\frac{ 1}{2^{k\epsilon}}
& \simeq   \frac 1 \epsilon . 
\end{align*}

\smallskip 
\textbf{Case 2.} This corresponds to the inequality \eqref{e:ii}.  
From \eqref{e:suff} and taking $L(t)=(\log_2 t)^\alpha$ instead of $\psi^{-1}$, 
\begin{align*}
c_\varphi\lesssim \sum_{k=1 }^\infty    \frac {1} {  k ^{\alpha }}    \lesssim \frac 1 {\alpha -1} . 
\end{align*}

\smallskip 
\textbf{Case 3.}  This corresponds to the inequality \eqref{e:iii}. 
From \eqref{e:suff} and taking $L(t)=\log_2t(\log_3 t)^\alpha$ instead of $\psi^{-1}$, 
\begin{align*}
 c_\varphi\lesssim\sum_{k=1 }^\infty  \frac {1 } {  k  (\log k) ^{\alpha}} \lesssim \frac 1 {\alpha  -1} .
\end{align*}

\medskip 
This proves Theorem \ref{t:HP} and Corollary \ref{c:borderline}. We also mention that if we apply our theorem to the case 
$\varphi(t)=t^r$, then $\psi^{-1}(t)\simeq t^{1/r'}$ and the constant that we obtain for $M_{\varphi(L)}=M_r$ is exactly $\log_1 r'$, as  in \eqref{e:Mr}. 

This last estimate \eqref{e:Mr}, combined with the reverse H\"older estimate \eqref{e:srh}, proves the estimate \eqref{e:logA}. 
We remark that an alternate proof of this result can be had by straight forward modification of the argument 
in the next section. Details are left to the reader.

\section{Square Functions: Proof of Theorem~\ref{t:sq}} 

Recall that it suffices to prove the weak-type bound for the sparse square function $ S f$ defined in \eqref{e:Sdef}, 
in which we can assume $ f$ is non-negative.  
The case $1\leq p<2$ is easy and contained in \cite{12114219}, as so our attention is on the case of 
 $p\geq 2$. The sparse collection of cubes $ \mathcal S$ is divided according to the approximate size of the average of $ f$.  
 For every integer $m$, define $\mathcal{S}_m$ by
\begin{equation*}
    \mathcal{S}_m = \{Q \in \mathcal{S}: \, 2^{-m-1} < \langle f \rangle_Q \leq 2^{-m} \}. 
\end{equation*}
The exceptional set  for $Q\in \mathcal S_m$ is defined relative to $ \mathcal S_m$ by 
\begin{equation*}
    E_m(Q) = Q \setminus \bigcup_{Q' \subsetneq Q,\, Q' \in \mathcal{S}_m} Q'.
\end{equation*}
By sparsity, we have that
 \begin{equation} \label{GoodControlEq}
        \langle f \mathbf{1}_{E_m(Q)} \rangle_Q \sim \langle f \rangle_Q.
    \end{equation}
Also, we set $S_m$ to be the square function associated with $\mathcal S_m$
\[
(S_m f) ^2 :=\sum_{Q\in \mathcal S_m} \langle f\rangle_Q^2\mathbf{1}_Q . 
\]
Thus, trivially,  
\[
(Sf) ^2 := \sum_{Q\in \mathcal S} \langle f\rangle_Q^2\mathbf{1}_Q=\sum_{m\in \mathbb Z} (S_m f)^2.
\]

For the case of moderate $ m$, namely $ 0 < m \leq C \log_1 [w] _{A_\infty}$, we have this estimate. 

\begin{lemma} For $p\geq 2$,
    \begin{equation}\label{Ap_bound}
        \| S_m f \|_{L^p(w)} \lesssim
        [w]_{A_p}^{1/2} \|f\|_{L^p(w)}
    \end{equation}
\end{lemma}

Here and below we consider the case of $ p\geq 2$.  The critical case is $ p=2$, but 
if we were to focus on this case, one would then have to appeal to an $ A_ \infty $ extrapolation 
argument for square functions, for which we do not have a clear cut reference. 
It is easier to simply prove the estimate for all $ 2\leq p < \infty $.

\begin{proof}
Rubio de Francia's extrapolation, as formulated in  \cite{MR2754896}*{Thm 3.1}, shows that the case of $ p=2$ 
implies the case of $ 2< p < \infty $.  
The case of $ p=2$ is a standard calculation.  We will use the dual weight $ \sigma = w ^{-1}$, and the 
standard trick of inserting $ 1 = [ w \cdot \sigma ] ^{1/2} $ inside a square.  
  By \eqref{GoodControlEq}, we can estimate 
  \begin{align*}
\sum_{Q \in \mathcal{S}_m} \langle f \rangle_Q^2   w (Q) & \lesssim 
  \sum_{Q \in \mathcal{S}_m} \langle f \mathbf{1}_{E_m(Q)} \rangle_Q^2 w(Q)
  \\
  & \leq  \sum_{Q \in \mathcal{S}_m} \langle f^2 \mathbf{1}_{E(Q)} w \rangle_Q \langle \sigma \rangle_Q w(Q) 
  \\
  &= \int f^2 \Bigl( \sum_{Q \in \mathcal{S}_m} \mathbf{1}_{E_m(Q)} \langle w \rangle_Q \langle \sigma \rangle_Q \Bigr) w 
  \leq [w]_{A_2} \int f^2 w. 
\end{align*}
\end{proof}

For large values of $ m$, this estimate is relevant. 
\begin{lemma}\label{Ainfty} For all integers $ m_0 > 0$, 
    \begin{equation}
     w\Bigl( \sum_{m=m_0}^\infty (S_m f)^2 > 1 \Bigr) \lesssim [w]_{A_p} \Bigl( \frac{[w]_{A_\infty}}{2^{m_0}} \Bigr)^p \|f\|_{L^p(w)}^p.
    \end{equation}
\end{lemma}

\begin{proof}
Write $ (S_m f) ^2 $ as  $ 2 ^{-2m} b_m$, where 
\begin{equation*}
b_m \leq\sum_{Q\in \mathcal S_m}\mathbf 1_{Q} 
\end{equation*}
and $ b_m$ are supported on the set $ B_m = \bigcup \{Q \;:\; Q\in \mathcal S_m ^{\ast} \}$. 
Here $ \mathcal S_m^{\ast}$ are the maximal cubes in $ \mathcal S_m$.  
On each cube $ Q\in \mathcal S_m$, the function $ b_m$ is locally exponentially integrable, by sparsity.  
By  the  sharp weak-type estimate for the maximal function, $ w (B_m) \lesssim 2 ^{pm} [ w] _{A_p} \lVert f\rVert _{L ^{p} (w)}^p$. 

We then estimate 
\begin{align*}
        w\Bigl( \sum_{m=m_0}^\infty (S_m f)^2 > 1 \Bigr) &= w\Bigl( \sum_{m=m_0}^\infty  2 ^{-2m} b_m > \sum_{m=m_0}^\infty 2^{m_0-m-1} \Bigr) \\
&\leq \sum_{m=m_0}^\infty  w(  b_m > 2^{m_0+m-1} ).
\end{align*}

We have a very good Lebesgue measure estimate for the \emph{Lebesgue measure} of the sets above. 
By sparsity, 
$
  \lvert  \{  b_m > 2^{m_0+m-1} \}\rvert \lesssim \exp( -C 2^{m_0+m} ) \lvert  B_m\rvert 
$.  Indeed, this estimate is uniform over the  cubes $Q\in \mathcal S_m ^{\ast}  $:  
Setting $ \beta (Q) :=  \{x \in Q \;:\;    b_m(x) > 2^{m_0+m-1} \}$, we have 
\begin{equation*}
  \lvert \beta (Q)\rvert \lesssim \exp( -C 2^{m_0+m} ) \lvert  Q\rvert 
\end{equation*}
This is converted to $ w$-measure, using the $A_\infty$ property of $A_p$ weights, together with the sharp reverse-H\"older estimate.  With $ r (w)$ as in \eqref{e:srh}, there holds 
    \begin{align*}
     \langle  w \mathbf 1_{ \beta (Q)} \rangle_Q 
     & \leq   \langle \mathbf 1_{\beta (Q)} \rangle ^{1/r (w)'} _{Q} \langle  w ^{r (w)} \rangle_Q ^{r (w)} 
    \\ & \lesssim 
        \Biggl[
        \frac { \lvert  \beta (Q)\rvert  }  { \lvert  Q\rvert }
        \Biggr]^{ (C[w]_{A_\infty})^{-1}}w(Q)
\lesssim w(Q) \exp\Bigl( - c \frac{2^{m_0+m} }{[w]_{A_\infty}} \Bigr). 
    \end{align*}
    Summing over the disjoint cubes in $ \mathcal S _{m} ^{\ast} $, we get
    \[
        w\Bigl( \sum_{m=m_0}^\infty (S_m f)^2 > 1 \Bigr) \lesssim [w]_{A_p} \|f\|_{L^p(w)}^p \sum_{m = m_0}^\infty 2^{mp} \exp\Bigl( -c \frac{2^{m_0+m} }{[w]_{A_\infty}} \Bigr).
    \]
The last sum is approximated by an integral to finish the proof.  
    \begin{align*}
        \sum_{m = m_0}^\infty 2^{mp} \exp\Bigl( -C \frac{2^{m_0+m} }{[w]_{A_\infty}} \Bigr) &\leq \int_{m_0}^\infty 2^{xp} \exp\Bigl( -c \frac{2^{m_0+x} }{[w]_{A_\infty}} \Bigr) \, dx 
        \\
        &\approx \int_{2^{m_0}}^\infty y^p \exp\Bigl( -c\frac{2 ^{m_0}}{[w]_{A_\infty}} y \Bigr) \, \frac{dy}{y} 
        \\
        &= \Bigl( \frac{[w]_{A_\infty}}{2^{m_0}} \Bigr)^p \int_{\frac{2^{2m_0}}{[w]_{A_\infty}}}^\infty y^p e^{-y} \, \frac{dx}{y} 
        \lesssim \Bigl( \frac{[w]_{A_\infty}}{2^{m_0}} \Bigr)^p. 
    \end{align*}
\end{proof}

The Lemmas are finished, and we can turn to the Theorem. 
Now, it suffices to estimate $ w ( S f >  \lambda )$, but it suffices to assume that $ \lambda =2$, and $ \lVert f\rVert _{L ^{p} (w)}=1$.  After division of the sparse collection $ \mathcal S$ into the subcollections $ \mathcal S _{m}$, for $ m\in \mathbb Z $, 
estimate for $ m_0  \approx \log_1 [w] _{A_ \infty }$,  
\[
    w((Sf)^2>2)\leq w (M f >1) + 
    w\Bigl( \sum_{m=1}^{m_0-1} (S_m f)^2  > 1 \Bigr)
    + w\Bigl( \sum_{m=m_0}^\infty (S_m f)^2  > 1 \Bigr).
\]
The first term is controlled by the sharp weak-type estimate for the maximal function, which yields an estimate 
smaller than what we claim for $ S f$. 
The second term is estimated by Chebysheff,   Minkowski's inequality as $ p\geq 2$, and the norm estimate from Lemma~\ref{Ap_bound}.  
\begin{align*}
    w\Bigl( \sum_{m=0}^{m_0-1} (S_m f)^2 > 1 \Bigr)
    &\leq \Bigl\|\sum_{m=0}^{m_0-1} (S_m f)^2 \Bigr\|_{L^{p/2}(w)}^{p/2} \\
    &\leq \Bigl( \sum_{m=0}^{m_0-1} \|(S_m f)^2 \|_{L^{p/2}(w)} \Bigr)^{p/2} \\
    &= \Bigl( \sum_{m=0}^{m_0-1} \|S_m f\|_{L^p(w)}^2 \Bigr)^{p/2}\leq \bigl( m_0 [w]_{A_p} \bigr)^{p/2}. 
\end{align*}
For the third term we can just use the estimate from Lemma \ref{Ainfty}:
\[
    w\Bigl( \sum_{m=m_0}^\infty (S_m f)^2 > 1 \Bigr) \leq [w]_{A_p} \Bigl( \frac{[w]_{A_\infty}}{2^{m_0}} \Bigr)^p.
\]

Combining these  estimates we get
\[
    \|Sf\|_{L^{p,\infty}(w)} \lesssim m_0^{\frac{1}{2}} [w]_{A_p}^{\frac{1}{2}} + [w]_{A_p}^{\frac{1}{p}} [w]_{A_\infty} 2^{-m_0} 
    \approx \bigl[[w]_{A_p} \log_1 [w] _{A_ \infty } \bigr] ^{\frac{1}{2}} 
\]
since  $m_0 \approx \log_1[w]_{A_\infty}$.


\begin{bibsection}  
\begin{biblist}




\bib{MR0133653}{article}{
   author={Benedek, A.},
   author={Calder{\'o}n, A. P.},
   author={Panzone, R.},
   title={Convolution operators on Banach space valued functions},
   journal={Proc. Nat. Acad. Sci. U.S.A.},
   volume={48},
   date={1962},
   pages={356--365},
   issn={0027-8424},
   review={\MR{0133653 (24 \#A3479)}},
}



\bib{MR0358205}{article}{
   author={Coifman, R. R.},
   author={Fefferman, C.},
   title={Weighted norm inequalities for maximal functions and singular
   integrals},
   journal={Studia Math.},
   volume={51},
   date={1974},
   pages={241--250},
   issn={0039-3223},
   review={\MR{0358205 (50 \#10670)}},
}


\bib{13125255}{article}{
   author = {{Criado}, A.},
   author={{Soria}, F.},
    title ={Muckenhoupt-Wheeden conjectures in higher dimensions},
   eprint = {1312.5255},
     year = {2013},
}

\bib{MR2797562}{book}{
   author={Cruz-Uribe, David V.},
   author={Martell, Jos{\'e} Maria},
   author={P{\'e}rez, Carlos},
   title={Weights, extrapolation and the theory of Rubio de Francia},
   series={Operator Theory: Advances and Applications},
   volume={215},
   publisher={Birkh\"auser/Springer Basel AG, Basel},
   date={2011},
   pages={xiv+280},
   isbn={978-3-0348-0071-6},
   review={\MR{2797562 (2012f:42001)}},
   doi={10.1007/978-3-0348-0072-3},
}

\bib{MR2754896}{article}{
   author={Duoandikoetxea, Javier},
   title={Extrapolation of weights revisited: new proofs and sharp bounds},
   journal={J. Funct. Anal.},
   volume={260},
   date={2011},
   number={6},
   pages={1886--1901},
   issn={0022-1236},
   review={\MR{2754896 (2012e:42027)}},
   doi={10.1016/j.jfa.2010.12.015},
}

\bib{MR1439553}{article}{
   author={Fefferman, R.},
   author={Pipher, J.},
   title={Multiparameter operators and sharp weighted inequalities},
   journal={Amer. J. Math.},
   volume={119},
   date={1997},
   number={2},
   pages={337--369},
   issn={0002-9327},
   review={\MR{1439553 (98b:42027)}},
}



\bib{MR0284802}{article}{
   author={Fefferman, C.},
   author={Stein, E. M.},
   title={Some maximal inequalities},
   journal={Amer. J. Math.},
   volume={93},
   date={1971},
   pages={107--115},
   issn={0002-9327},
   review={\MR{0284802 (44 \#2026)}},
}
		
\bib{MR3092729}{article}{
   author={Hyt{\"o}nen, T.},
   author={P{\'e}rez, C.},
   title={Sharp weighted bounds involving $A_\infty$},
   journal={Anal. PDE},
   volume={6},
   date={2013},
   number={4},
   pages={777--818},
   issn={2157-5045},
   review={\MR{3092729}},
   doi={10.2140/apde.2013.6.777},
}



\bib{MR3327006}{article}{
   author={Hyt{\"o}nen, T.},
   author={P{\'e}rez, C.},
   title={The $L(\log L)^\epsilon$ endpoint estimate for maximal singular
   integral operators},
   journal={J. Math. Anal. Appl.},
   volume={428},
   date={2015},
   number={1},
   pages={605--626},
   issn={0022-247X},
   review={\MR{3327006}},
   doi={10.1016/j.jmaa.2015.03.017},
}


\bib{orlicz}{book}{
  author={Krasnosel'skii, M. A.},
  author={Rutickii, Y. B.},
  title={Convex Functions and Orlicz Spaces},
  date={1961} 
  publisher={P. Noordhoff Ltd.}
  address={Groningen}
}


\bib{150105818}{article}{
  author={Lacey, M. T.},
  title={An elementary proof of the $A_2$ Bound},
  date={2015},
  eprint={http://arxiv.org/abs/1501.05818},
  journal={Israel J. Math., to appear}, 
}


\bib{12114219}{article}{
  author={Lacey, M. T.},
  author={Scurry, J.}
  title={Weighted Weak Type Estimates for Square Functions},
  date={2012},
  eprint={http://arxiv.org/abs/1211.4219},
}

\bib{MR2770437}{article}{
   author={Lerner, Andrei K.},
   title={Sharp weighted norm inequalities for Littlewood-Paley operators
   and singular integrals},
   journal={Adv. Math.},
   volume={226},
   date={2011},
   number={5},
   pages={3912--3926},
   issn={0001-8708},
   review={\MR{2770437 (2012c:42048)}},
   doi={10.1016/j.aim.2010.11.009},
}

\bib{MR2427454}{article}{
   author={Lerner, A. K.},
   author={Ombrosi, S.},
   author={P{\'e}rez, C.},
   title={Sharp $A\sb 1$ bounds for Calder\'on-Zygmund operators and the
   relationship with a problem of Muckenhoupt and Wheeden},
   journal={Int. Math. Res. Not. IMRN},
   date={2008},
   number={6},
   pages={Art. ID rnm161, 11},
   issn={1073-7928},
   review={\MR{2427454 (2009e:42030)}},
   doi={10.1093/imrn/rnm161},
}

\bib{MR2480568}{article}{
   author={Lerner, A. K.},
   author={Ombrosi, S.},
   author={P{\'e}rez, C.},
   title={$A\sb 1$ bounds for Calder\'on-Zygmund operators related to a
   problem of Muckenhoupt and Wheeden},
   journal={Math. Res. Lett.},
   volume={16},
   date={2009},
   number={1},
   pages={149--156},
   issn={1073-2780},
   review={\MR{2480568 (2010a:42052)}},
   doi={10.4310/MRL.2009.v16.n1.a14},
}

\bib{problems}{article}{
   author={Muckenhoupt, Benjamin},
   title={Problems},
   conference={
      title={Harmonic analysis},
   },
   book={
      series={Contemp. Math.},
      volume={411},
      publisher={Amer. Math. Soc., Providence, RI},
   },
   date={2006},
   pages={131-135},
}


\bib{MR1260114}{article}{
   author={P{\'e}rez, C.},
   title={Weighted norm inequalities for singular integral operators},
   journal={J. London Math. Soc. (2)},
   volume={49},
   date={1994},
   number={2},
   pages={296--308},
   issn={0024-6107},
   review={\MR{1260114 (94m:42037)}},
   doi={10.1112/jlms/49.2.296},
}


\bib{MR2799801}{article}{
   author={Reguera, M. C.},
   title={On Muckenhoupt-Wheeden conjecture},
   journal={Adv. Math.},
   volume={227},
   date={2011},
   number={4},
   pages={1436--1450},
   issn={0001-8708},
   review={\MR{2799801 (2012d:42038)}},
   doi={10.1016/j.aim.2011.03.009},
}

\bib{MR2923171}{article}{
   author={Reguera, M. C.},
   author={Thiele, C.},
   title={The Hilbert transform does not map $L\sp 1(Mw)$ to $L\sp
   {1,\infty}(w)$},
   journal={Math. Res. Lett.},
   volume={19},
   date={2012},
   number={1},
   pages={1--7},
   issn={1073-2780},
   review={\MR{2923171}},
   doi={10.4310/MRL.2012.v19.n1.a1},
}

\bib{MR663793}{article}{
   author={Rubio de Francia, Jos{\'e} Luis},
   title={Factorization and extrapolation of weights},
   journal={Bull. Amer. Math. Soc. (N.S.)},
   volume={7},
   date={1982},
   number={2},
   pages={393--395},
   issn={0273-0979},
   review={\MR{663793 (83i:42016)}},
   doi={10.1090/S0273-0979-1982-15047-9},
}

\bib{MR745140}{article}{
   author={Rubio de Francia, Jos{\'e} L.},
   title={Factorization theory and $A\sb{p}$ weights},
   journal={Amer. J. Math.},
   volume={106},
   date={1984},
   number={3},
   pages={533--547},
   issn={0002-9327},
   review={\MR{745140 (86a:47028a)}},
   doi={10.2307/2374284},
}

\bib{MR736248}{article}{
   author={Rubio de Francia, J. L.},
   author={Ruiz, F. J.},
   author={Torrea, J. L.},
   title={Les op\'erateurs de Calder\'on-Zygmund vectoriels},
   language={French, with English summary},
   journal={C. R. Acad. Sci. Paris S\'er. I Math.},
   volume={297},
   date={1983},
   number={8},
   pages={477--480},
   issn={0249-6291},
   review={\MR{736248 (85h:42024)}},
}




\bib{wilson1}{article}{
   author={Wilson, Michael},
   title={The intrinsic square function},
   journal={Rev. Mat. Iberoam.},
   volume={23},
   date={2007},
   number={3},
   pages={771--791},
   issn={0213-2230},
   review={\MR{2414491 (2009e:42039)}},
   doi={10.4171/RMI/512},
}





\end{biblist}
\end{bibsection}
\end{document}